\documentclass[leqno,11pt]{amsart}
\usepackage{amsmath,amssymb}
\usepackage{color}
\usepackage{comment}

\theoremstyle{plain}
\newtheorem{theorem}{Theorem}[section]

\theoremstyle{definition}

\newtheorem{remark}[theorem]{Remark}

\newtheorem{question}[theorem]{Question}

\def\Rn{\mathbb R\sp n}
\def\R{\mathbb R}

\def\m{_{\max}}

\newtoks\by
\newtoks\paper
\newtoks\book
\newtoks\jour
\newtoks\yr
\newtoks\pages
\newtoks\vol
\newtoks\publ
\newtoks\eds
\newtoks\proc
\def\ota{{\hbox{???}}}
\def\cLear{\by=\ota\paper=\ota\book=\ota\jour=\ota\yr=\ota
\pages=\ota\vol=\ota\publ=\ota}
\def\endpaper{\the\by, \textit{\the\paper},
{\the\jour} \textbf{\the\vol} (\the\yr), \the\pages.\cLear}
\def\endbook{\the\by, \textit{\the\book}, \the\publ.\cLear}
\def\endprep{\the\by, \textit{\the\paper}, \the\jour.\cLear}
\def\endproc{\the\by, \textit{\the\paper}, \the\publ, \the\pages.\cLear}
\def\name#1#2{#1 #2}
\def\et{ and }

\numberwithin{equation}{section}

\begin{document}

\title[On the necessity of bump conditions]
{On the necessity of bump conditions for the two-weighted maximal inequality}

\author{Lenka Slav\'ikov\'a}

\address{Department of Mathematical Analysis\\
Faculty of Mathematics and Physics\\
Charles University\\
Sokolovsk\'a~83\\
186~75 Praha~8\\
Czech Republic} \email{slavikova@karlin.mff.cuni.cz}

\subjclass[2010]{42B25, 42B35} \keywords{Bump condition, two-weighted inequality, Hardy-Littlewood maximal operator}

\thanks{This research was partly supported by the grant P201-13-14743S of the Grant Agency of the Czech Republic.}

\begin{abstract}
We study the necessity of bump conditions for the boundedness of the Hardy-Littlewood maximal operator $M$ from $L^p(v)$ into $L^p(w)$, where $1<p<\infty$. The conditions in question are obtained by replacing the average of $\sigma=v^{-\frac{1}{p-1}}$ in the Muckenhoupt $A_p$-condition by an average with respect to certain Banach function space, and are known to be sufficient for the two-weighted maximal inequality. We show that these conditions are in general not necessary for the boundedness of $M$ from $L^p(v)$ into $L^p(w)$.
\end{abstract}

\maketitle

\section{Introduction and statement of the result}\label{S:intro}

The Hardy--Littlewood maximal operator $M$ is defined for every measurable function $f$ on $\Rn$ by
$$
Mf(x)=\sup_{x\in Q} \frac{1}{|Q|} \int_Q |f|, \quad x\in \Rn,  
$$
where the supremum is taken over all cubes $Q$ containing $x$. By a ``cube'' we always mean a compact cube with sides parallel to coordinate axes.

Assume that $1<p<\infty$. A longstanding open problem in harmonic analysis is to characterize those couples $(w,v)$ of nonnegative measurable functions, called weights in the sequel, which satisfy the inequality
\begin{equation}\label{E:boundedness_M}
\int_{\Rn} w (Mf)^p\leq C \int_{\Rn} vf^p
\end{equation} 
for all nonnegative measurable functions $f$ and some positive constant $C$. In what follows, let us exclude the trivial case when $w=0$ a.e.\ and employ the usual notation $\sigma=v^{-\frac{1}{p-1}}$.

In the special case when the two weights $(w,v)$ coincide, inequality~\eqref{E:boundedness_M} was characterized by Muckenhoupt~\cite{M}. He showed that~\eqref{E:boundedness_M} holds with $v=w$ if and only if $w$ satisfies the $A_p$-condition
\begin{equation}\label{E:ap}
\sup_{Q} \left(\frac{1}{|Q|} \int_Q w\right)\left(\frac{1}{|Q|}\int_Q \sigma\right)^{p-1} <\infty.
\end{equation}
We note that throughout this paper, the notation $\sup_{Q}$ means that the supremum is taken over all cubes $Q$ in $\Rn$. 

The situation is much more complicated in the two-weighted case, since the $A_p$-condition~\eqref{E:ap} is still necessary for~\eqref{E:boundedness_M},
but it is not sufficient any more (see, e.g.,~\cite[Chapter 4, Example 1.15]{GR}). A solution to the two-weighted problem was given by Sawyer~\cite{Sa}, who showed that~\eqref{E:boundedness_M} holds if and only if there is a positive constant $C$ such that
\begin{equation}\label{E:sawyer}
\int_Q w (M(\chi_Q\sigma))^p \leq C\int_Q \sigma <\infty
\end{equation}
for every cube $Q$. Condition~\eqref{E:sawyer} still involves the operator $M$ itself, and it remains an open problem to find a characterization of~\eqref{E:boundedness_M} which is closer in form to~\eqref{E:ap}, and which thus might be easier to use in applications. 
An important result in this direction was obtained by Neugebauer~\cite{N} and improved subsequently by P\'erez~\cite{P} and P\'erez and Rela~\cite{PR}. In these papers, the authors found strengthenings of the $A_p$-condition~\eqref{E:ap} that are already sufficient for~\eqref{E:boundedness_M}. The necessity of these ``bump'' conditions is however not discussed. We show in this paper that none of the sufficient conditions from~\cite{N, P, PR} is necessary for~\eqref{E:boundedness_M}.

Before stating our result we need to describe in some detail the results from~\cite{N, P, PR}. Neugebauer~\cite{N} showed that if
\begin{equation}\label{E:neugebauer}
\sup_{Q} \left(\frac{1}{|Q|} \int_Q w^{r}\right)^{\frac{1}{r}} \left(\frac{1}{|Q|}\int_Q \sigma^{r}\right)^{\frac{p-1}{r}} <\infty
\end{equation}
for some $r>1$ then~\eqref{E:boundedness_M} is fulfilled. P\'erez~\cite{P} found a way how to weaken the sufficient condition~\eqref{E:neugebauer}. He noticed that in order to obtain~\eqref{E:boundedness_M} one just needs to replace in a suitable way the integral average of $\sigma$ in~\eqref{E:ap}. While in~\cite{N}, this integral average was replaced by the average with respect to certain Lebesgue space, in~\cite{P} it was shown that more general spaces of measurable functions, the so called Banach function spaces, can be used in this connection as well. 

We note that the reader can find a precise definition of the notion ``Banach function space", together with many of its properties and several specific examples, in~\cite{BS}; here we just recall that a Banach function space $X$ is a Banach space containing measurable functions on $\Rn$ whose norm $\|\cdot\|_X$ is induced by a so called function norm $\rho$. This notion means that $\rho$ is a nonnegative functional defined on the set of all nonnegative measurable functions on $\Rn$ which fulfills the properties of a norm, is monotone ($0\leq f\leq g$ a.e.\ implies $\rho(f) \leq \rho(g)$) and satisfies the Fatou property ($0\leq f_n \nearrow f$ a.e.\ implies $\rho(f_n) \nearrow \rho(f)$). Two nontriviality assumptions are also required ($\rho$ is finite on characteristic functions of sets of finite measure, and is bounded from below by a constant multiple of the $L^1$-norm on those sets). The space $X$ then consists exactly of those functions $f$ for which $\rho(|f|)< \infty$, and $\|f\|_X=\rho(|f|)$ in this situation. In fact, it will be convenient for us to admit a slight abuse of notation and write $\|f\|_{X}:=\rho(|f|)$ also in the case when $\rho(|f|)$ is infinite.

To each Banach function space $X$ there corresponds its associate space $X'$, which is another Banach function space induced by the function norm $\rho'$ given by
$$
\rho'(f)=\sup_{\rho(g)\leq 1} \int_{R^n} fg. 
$$

The $X$-average of a measurable function $f$ over a cube $Q$ is defined in~\cite{P} by
$$
\|f\|_{X,Q}=\|\tau_{\ell(Q)} f\chi_Q\|_{X},
$$
where $\tau_\delta$ denotes, for $\delta>0$, the dilation operator $\tau_\delta f(x)=f(\delta x)$, and $\ell(Q)$ stands for the sidelength of the cube $Q$. The maximal operator $M_X$ is then given by
$$
M_X f(x)=\sup_{x\in Q} \|f\|_{X,Q}, \quad x\in \Rn.
$$
Notice that if $X=L^1$ then $M_X$ coincides with the classical Hardy-Littlewood maximal operator $M$. 

The sufficient condition for~\eqref{E:boundedness_M} proved in~\cite{P} has the form
\begin{equation}\label{E:ap_new}
\sup_{Q} \left(\frac{1}{|Q|} \int_Q w\right) \|\sigma^{\frac{1}{p'}}\|^p_{X,Q} <\infty,
\end{equation}
where $p'=\frac{p}{p-1}$ and $X$ is any Banach function space whose associate space $X'$ fulfills
\begin{equation}\label{E:mx}
\int_{\Rn} (M_{X'}f)^p \leq C \int_{\Rn} f^p
\end{equation}
for all nonnegative measurable functions $f$ and some positive constant $C$. A basic example of a Banach function space for which this result can be applied is the Lebesgue space $L^q$ with $q>p'$. The strength of the result lies, however, in more delicate Banach function spaces, such as, e.g., the Orlicz space $L^A$ corresponding to the Young function $A(t)=t^{p'} \log^{\gamma}(1+t)$ with $\gamma>p'-1$ (for the definition of a Young function and of an Orlicz space, see, e.g.,~\cite{BS}).

Condition~\eqref{E:mx} can be weakened if we allow its dependence on $\sigma$. Namely, the following implication holds: if $X$ is a Banach function space such that~\eqref{E:ap_new} is fulfilled and there is a positive constant $C$ for which
\begin{equation}\label{E:rela}
\int_Q (M_{X'}(\sigma^{\frac{1}{p}}\chi_Q))^p \leq C \int_Q \sigma <\infty
\end{equation}
for every cube $Q$, then~\eqref{E:boundedness_M} holds. This was proved by P\'erez and Rela~\cite{PR} as a consequence of the Sawyer characterization of the two-weighted maximal inequality. We note that the result in~\cite{PR} is restricted only to Orlicz spaces, however, it is easy to observe that the proof given here works equally well for an arbitrary Banach function space over $\Rn$. Moreover, the paper~\cite{PR} gives even a quantitative version of this result which is shown to hold, at least for Orlicz spaces, not only in the Euclidean setting, but also in the more general context of spaces of homogeneous type.

Notice that the new condition~\eqref{E:rela} is in many situations much weaker than the original one~\eqref{E:mx}. For instance, one can easily observe that~\eqref{E:mx} is not valid when $X'=L^p$, while~\eqref{E:rela} holds with $X'=L^p$ if and only if
\begin{equation}\label{E:a_infty}
\int_Q (M_{L^p}(\sigma^{\frac{1}{p}}\chi_Q))^p
=\int_Q M(\sigma \chi_Q) \leq C\int_Q \sigma <\infty
\end{equation}
for all cubes $Q$. It was shown by Fujii~\cite{F} and rediscovered later by Wilson~\cite{W} that the validity of condition~\eqref{E:a_infty} is equivalent to the fact that $\sigma$ is an $A_\infty$-weight, that is, a locally integrable weight which satisfies the one-weighted $A_p$-condition for some $p>1$.

As observed before, it is a relevant question to ask whether the previously mentioned sufficient condition (which is the weakest one of those appearing in~\cite{N, P, PR}) is also necessary for~\eqref{E:boundedness_M}.

\begin{question}\label{Q:question}
Given a couple $(w,v)$ of weights satisfying~\eqref{E:boundedness_M}, is it true that there is a Banach function space $X$ fulfilling~\eqref{E:ap_new} and~\eqref{E:rela}?
\end{question}

We notice that the answer to this question is positive whenever $\sigma$ is an $A_\infty$-weight. Indeed, in this situation it suffices to take $X=L^{p'}$. We already know that~\eqref{E:rela} is then fulfilled. Further, condition~\eqref{E:ap_new} is in this case just the standard $A_p$-condition, which is well known to be necessary for~\eqref{E:boundedness_M}. In fact, according to the reverse H\"older inequality (see, e.g.,~\cite[Chapter 4, Lemma 2.5]{GR}), condition~\eqref{E:boundedness_M} implies even~\eqref{E:ap_new} with $X=L^{p'+\varepsilon}$ for some $\varepsilon>0$, depending on $\sigma$. Since the space $X=L^{p'+\varepsilon}$ satisfies not only~\eqref{E:rela}, but also the stronger condition~\eqref{E:mx}, one can obtain even a better conclusion in this case.

The interesting problem is whether a similar result holds without the $A_\infty$-assumption. We show that this is not the case in general.

Given $x\in \Rn$, we shall denote by $|x|_{\max}$ the maximum norm of $x$, that is, if $x=(x_1,\dots,x_n)$ then $|x|_{\max}=\max_{i=1,\dots,n} |x_i|$. 

\begin{theorem}\label{T:main}
Let $1<p<\infty$, and let
\begin{align*}
&w(x)=\frac{|x|_{\max}^{n(p-1)}}{(1+\log_+ |x|_{\max})^p}, \\
&\sigma(x)=\frac{1}{|x|_{\max}^n (1+\log_+ \frac{1}{|x|_{\max}})^{p'}} \quad \textup{for a.e. } x\in \Rn.
\end{align*}
Then the couple $(w,v)$, with $v=\sigma^{1-p}$, fulfills~\eqref{E:boundedness_M}, but there is no Banach function space $X$ for which~\eqref{E:ap_new} and~\eqref{E:rela} hold simultaneously.
\end{theorem}

\begin{remark}
Assume that $\alpha \in (0,n)$ and $\beta \in \R$, and set $v=\sigma^{1-p}$, where
$$
\sigma(x)=\frac{1}{|x|_{\max}^\alpha (1+\log_+ \frac{1}{|x|_{\max}})^{\beta}} \quad \textup{for a.e. } x\in \Rn.
$$
Then the answer to Question~\ref{Q:question} is positive, regardless of what $w$ is. This follows from the fact that $\sigma$ is an $A_\infty$ weight, combined with our previous observations.
\end{remark}

\section{Proof of Theorem~\ref{T:main}}\label{S:proofs}

We devote this section to the proof of Theorem~\ref{T:main}. Throughout the proof, we shall denote by $Q(x,r)$ the cube centered in $x$ and with sidelength $r$. 

\begin{proof}[Proof of Theorem~\ref{T:main}]
We first observe that 
$$
M\sigma(x) \approx \frac{1+\left|\log \frac{1}{|x|_{\max}}\right|}{|x|\m^n (1+\log_+ \frac{1}{|x|\m})^{p'}}, \quad x\in \Rn \setminus \{0\},
$$
and therefore,
$$
(M\sigma)^p(x) w(x) \approx \sigma(x) \quad \textup{for a.e. } x\in \Rn.
$$
In both cases, ``$\approx$" denotes the equivalence up to multiplicative constants depending on $p$ and $n$. Hence, for any cube $Q$,
$$
\int_Q (M(\chi_Q \sigma))^p(x) w(x)\,dx 
\leq \int_Q (M\sigma)^p(x) w(x)\,dx
\approx \int_Q \sigma (x)\,dx,
$$
and Sawyer's result yields that the couple $(w,v)$ satisfies~\eqref{E:boundedness_M}. 

Let $X$ be any Banach function space. Given $b\in (0,2)$, we have
\begin{align*}
\left\|\sigma^{\frac{1}{p}}\right\|_{X', Q(0,b)}
&=\left\|\left(\sigma^{\frac{1}{p}} \chi_{Q(0,b)}\right)(b y)\right\|_{X'}\\
&=\frac{1}{b^{\frac{n}{p}}}\left\|\frac{\chi_{Q(0,1)}(y)}{|y|\m^{\frac{n}{p}}\left(1+\log_+ \frac{1}{b|y|\m}\right)^{\frac{p'}{p}}} \right\|_{X'}\\
&\geq \frac{1}{b^{\frac{n}{p}}} \left\|\frac{\chi_{Q(0,1)\setminus Q(0,\frac{b}{2})}(y)}{|y|\m^{\frac{n}{p}}\left(1+\log_+ \left(\frac{2}{b}\right)^2\right)^{\frac{p'}{p}}} \right\|_{X'}\\
&= \frac{1}{b^{\frac{n}{p}}\left(1+2\log_+ \frac{2}{b}\right)^{\frac{p'}{p}}} \left\|\frac{\chi_{Q(0,1)\setminus Q(0,\frac{b}{2})}(y)}{|y|\m^{\frac{n}{p}}} \right\|_{X'}\\
&\geq \frac{1}{2^{\frac{p'}{p}}b^{\frac{n}{p}} (1+\log_+ \frac{2}{b})^{\frac{p'}{p}}} \left\|\frac{\chi_{Q(0,1)\setminus Q(0,\frac{b}{2})}(y)}{|y|\m^{\frac{n}{p}}}\right\|_{X'}.
\end{align*}
Thus, for any $a\in (0,2)$,
\begin{align}\label{E:est}
\int_{Q(0,a)} \left(M_{X'}(\sigma^{\frac{1}{p}} \chi_{Q(0,a)})\right)^p(x)\,dx
&\geq \int_{Q(0,a)} \left\|\sigma^{\frac{1}{p}}\right\|_{X', Q(0,2|x|\m)}^p\,dx\\ \nonumber
&\geq \int_{Q(0,a)} \frac{1}{2^{p'+n}|x|\m^{n} (1+\log_+ \frac{1}{|x|\m})^{p'}} \left\|\frac{\chi_{Q(0,1)\setminus Q(0,|x|\m)}(y)}{|y|\m^{\frac{n}{p}}}\right\|_{X'}^p\,dx\\ \nonumber
&\geq \frac{1}{2^{p'+n}} \left\|\frac{\chi_{Q(0,1)\setminus Q(0,\frac{a}{2})}(y)}{|y|\m^{\frac{n}{p}}}\right\|_{X'}^p  \int_{Q(0,a)}\frac{\,dx}{|x|\m^{n} (1+\log_+ \frac{1}{|x|\m})^{p'}}\\
&=\frac{1}{2^{p'+n}} \left\|\frac{\chi_{Q(0,1)\setminus Q(0,\frac{a}{2})}(y)}{|y|\m^{\frac{n}{p}}}\right\|_{X'}^p \int_{Q(0,a)} \sigma(x)\,dx. \nonumber
\end{align}

Assume that $X$ fulfills~\eqref{E:rela}. Then there is a constant $C>0$, independent of $a\in (0,2)$, such that
\begin{equation}\label{E:assumption}
\int_{Q(0,a)} \left(M_{X'}(\sigma^{\frac{1}{p}} \chi_{Q(0,a)})\right)^p(x)\,dx \leq C \int_{Q(0,a)} \sigma(x)\,dx.
\end{equation}
Since $\int_{Q(0,a)} \sigma(x)\,dx$ is positive and finite, a combination of~\eqref{E:est} and~\eqref{E:assumption} yields that
$$
\left\|\frac{\chi_{Q(0,1)\setminus Q(0,\frac{a}{2})}(y)}{|y|\m^{\frac{n}{p}}}\right\|_{X'} \leq 2^{\frac{p'+n}{p}} C^{\frac{1}{p}} =:D.
$$
Passing to limit when $a$ tends to $0$ and using the Fatou property of $\|\cdot\|_X$, we obtain
\begin{equation}\label{E:norm}
\left\|\frac{\chi_{Q(0,1)}(y)}{|y|\m^{\frac{n}{p}}}\right\|_{X'} \leq D.
\end{equation}

To get a contradiction, assume that condition~\eqref{E:ap_new} is satisfied as well. Since 
\begin{align*}
\int_{Q(0,1)} w(x)\,dx \|\sigma^{\frac{1}{p'}}\|_{X,Q(0,1)}^p
\leq \sup_{Q} \left(\frac{1}{|Q|} \int_Q w(x)\,dx\right) \|\sigma^{\frac{1}{p'}}\|^p_{X,Q}
<\infty
\end{align*}
and $\int_{Q(0,1)} w(x)\,dx$ is clearly positive, the function $\sigma^{\frac{1}{p'}} \chi_{Q(0,1)}$ belongs to $X$. However, by~\eqref{E:norm} and by the identity $X=(X')'$ (see, e.g.,~\cite[Chapter 1, Theorem 2.7]{BS}), we have
\begin{align*}
\|\sigma^{\frac{1}{p'}}\|_{X,Q(0,1)} 
&=\sup_{\|f\|_{X'}\leq 1} \int_{Q(0,1)} \sigma^{\frac{1}{p'}}(x) |f(x)|\,dx\\
&\geq \frac{1}{D} \int_{Q(0,1)} \frac{\sigma^{\frac{1}{p'}}(x)}{|x|\m^{\frac{n}{p}}}\,dx\\
&=\frac{1}{D} \int_{Q(0,1)} \frac{\,dx}{|x|\m^n (1+\log_+ \frac{1}{|x|\m})} =\infty,
\end{align*}
a contradiction.
Thus, conditions~\eqref{E:ap_new} and~\eqref{E:rela} cannot be fulfilled simultaneously. The proof is complete.
\end{proof}

\section*{Acknowledgements}

I would like to thank Carlos P\'erez for fruitful discussions on two-weighted inequalities and, in particular, for suggesting to me the problem of the necessity of bump conditions. I am also grateful to Lubo\v s Pick for careful reading of this paper.


\begin{thebibliography}{999}

\bibitem{BS}
\by{\name{C.}{Bennett}\et\name{R.}{Sharpley}}
\book{Interpolation of operators}
\publ{Pure and Applied Mathematics Vol.\ 129, Academic
Press, Boston 1988}
\endbook

\bibitem{F}
\by{\name{N.}{Fujii}}
\paper{Weighted bounded mean oscillation and singular integrals} \jour{Math. Japon.} \vol{22} \yr{1977/78} \pages{no. 5, 529--534}
\endpaper

\bibitem{GR}
\by{\name{J.}{Garc{\'{\i}}a-Cuerva}\et\name{J. L.}{Rubio de Francia}}
\book{Weighted norm inequalities and related topics}
\publ{North-Holland Mathematics Studies Vol.\ 116, North-Holland Publishing Co., Amsterdam, 1985. Notas de Matem{\'a}tica [Mathematical Notes], 104}
\endbook

\bibitem{M}
\by{\name{B.}{Muckenhoupt}}
\paper{Weighted norm inequalities for the {H}ardy maximal function} \jour{Trans. Amer. Math. Soc.} \vol{165} \yr{1972} \pages{207--226}
\endpaper

\bibitem{N}
\by{\name{C. J.}{Neugebauer}}
\paper{Inserting {$A_{p}$}-weights} \jour{Proc. Amer. Math. Soc.} \vol{87} \yr{1983} \pages{no. 4, 644--648}
\endpaper

\bibitem{P}
\by{\name{C.}{P\'erez}}
\paper{On sufficient conditions for the boundedness of the
              {H}ardy-{L}ittlewood maximal operator between weighted {$L^p$}-spaces with different weights} \jour{Proc. London Math. Soc. (3)} \vol{71} \yr{1995} \pages{no. 1, 135--157}
\endpaper

\bibitem{PR}
\by{\name{C.}{P\'erez}\et\name{E.}{Rela}}
\paper{A new quantitative two weight theorem for the
              {H}ardy-{L}ittlewood maximal operator} \jour{Proc. Amer. Math. Soc.} \vol{143} \yr{2015} \pages{no. 2, 641--655}
\endpaper

\bibitem{Sa}
\by{\name{E. T.}{Sawyer}}

\paper{A characterization of a two-weight norm inequality for maximal
              operators} \jour{Studia Math.} \vol{75} \yr{1982} \pages{no. 1, 1--11}
\endpaper
\bibitem{W}
\by{\name{J. M.}{Wilson}}
\paper{Weighted inequalities for the dyadic square function without
              dyadic {$A_\infty$}} \jour{Duke Math. J.} \vol{55} \yr{1987} \pages{no. 1, 19--50}
\endpaper


\end{thebibliography}
\end{document}